\documentclass[12pt,a4paper]{amsart}
\usepackage{latexsym,amssymb,amscd,amsmath,mathrsfs}

\newtheorem{theorem}{Theorem}
\newtheorem{lemma}{Lemma}

\newtheorem{corol}{Corollary}
\theoremstyle{definition}
\newtheorem{exam}{Example}

\def\bwr{\mathop{\text{\large$\wr$}}}

\def\lst#1#2{ #1_1 , #1_2 , \dots , #1_{#2} }

\def\setsuch#1#2{\left\{\,#1\mid #2\,\right\}}
\def\gnrsuch#1#2{\langle\,#1\mid #2\,\rangle}
\def\={\setminus}	\def\0{\emptyset}

\def\fX{\mathbf x}	\def\fY{\mathbf y}
\def\bR{\mathbf R}	\def\bS{\mathbf S}
\def\8{\infty}

\begin{document}
 
 \title{Generators and relations for wreath products}
 \author{Yu.~A.~Drozd}
 \address{Institute of Mathematics, National Academy of Sciences of Ukraine,
 Tereschenkivska str. 3, 01601 Kiev, Ukraine}
 \email{drozd@imath.kiev.ua}
 \urladdr{www.imath.kiev.ua/$~$drozd}
 \author{R.~V.~Skuratovski}
 \address{Kiev National Taras Shevchenko University,
 Department of Mechanics and Mathematics, Volodymyrska str. 64,
 01033 Kiev, Ukraine}
 \email{ruslcomp@mail.ru}
  \begin{abstract}
  Generators and defining relations for wreath products of groups are given. Under some condition
 (conormality of the generators) they are minimal. In particular, it is just the case for the Sylow
 subgroups of the symmetric groups.
 \end{abstract}
\maketitle

 Let $G,H$ be two groups. Denote by $H^G$ the group of all maps $f:G\to H$ with \emph{finite support},
 i.e. such that $f(x)=1$ for all but a finite set of elements of $G$. Recall that their (\emph{restricted regular})
 \emph{wreath product} $W=H\wr G$ is defined as the semidirect product $H^G\rtimes G$ with the
 natural action of $G$ on $H^G$: $f^g(a)=f(ag)$ \cite[p.\,175]{ro}. We are going to find a set of generators
 and relations for $H\wr G$ knowing those for $G$ and $H$. Then we shall extend this result to the
 \emph{multiple wreath products} $\bwr_{k=1}^nG_k=(\dots((G_1\wr G_2)\wr G_3)\dots)\wr G_n$.
 
 If $\fX=\{\lst xn\}$ are generators for $G$ and $\bR=\{\lst Rm\}$ are defining relations for this
 set of generators, we write 
 $$G:=\gnrsuch{\lst xn}{\lst Rm}$$
 or $G:=\gnrsuch{\fX}{\bR}$. A presentation is called \emph{minimal} if neither of the generators $\lst xn$
 nor of the relations $\lst Rm$ can be excluded. We call the set of generators $\fX$ \emph{conormal} if neither
 element $x\in\fX$ belongs to the normal subgroup $N_x$ generated by all $y\in\fX\=\{x\}$. For instance, any
 minimal set of generators of a finite $p$-group $G$ is conormal since their images are linear independent in the
 factorgroup $G/G^p[G,G]$ \cite[Theorem~5.48]{ro}. 
 
 \begin{theorem}\label{main}
 Let $G:=\gnrsuch{\fX}{\bR(\fX)},\ H:=\gnrsuch{\fY}{\bS(\fY)}$ be presentations of $G$ and $H$. Choose
 a subset $T\subseteq G$ such that $T\cap T^{-1}=\0$ and $T\cup T^{-1}=G\=\{1\}$, where
 $T^{-1}=\setsuch{t^{-1}}{t\in T}$. Then the wreath product $W=H\wr G$ has a presentation of the form 
 \begin{multline}\label{e1}
 W:=\gnrsuch{\fX,\fY}{\bR(\fX),\,\bS(\fY),\,[y,t^{-1}zt]=1  \text{\emph{ for all }}y,z\in\fY,\,t\in T}. 
 \end{multline}
 If the given presentations of $G$ and $H$ are minimal and the set of generators $\fY$ is conormal, the
 presentation \eqref{e1} is minimal as well.
 \end{theorem}
 
 \begin{theorem}\label{mult}
 Let $G_i:=\gnrsuch{\fX_i}{\bR_i(\fX_i)}$ be presentations of the groups $G_i\ (1\le i\le m)$. For
 $1<i\le m$ choose a subset $T_i\subseteq G_i$ such that $T_i\cap T_i^{-1}=\0$ and
 $T_i\cup T_i^{-1}=G_i\=\{1\}$. Then the wreath product $W=\bwr_{i=1}^mG_i$ has a presentation of the form 
 \begin{multline}\label{e2}
 W:=\gnrsuch{\fX_i,\ 1\le i\le m}{\bR_i(\fX_i),\ 1\le i\le m;\,[x,t^{-1}yt]=1\\
 \text{\emph{ for all }} x,y\in\mbox{\boldmath$\cup$}_{i<j}\,\fX_i,\,t\in T_j}. 
 \end{multline}
  If all given presentations of $G_i$ are minimal and the sets of generators $\fX_i\ ( 1\le i<n)$
 are conormal, the presentation \eqref{e2} is minimal as well.
  \end{theorem}
 
 In what follows, we keep the notations of Theorem~\ref{main}.
 Note that $H^G=\bigoplus_{a\in G}H(a)$, where $H(a)$ is a copy of the group $H$; the elements of $H(a)$
 will be denoted by $h(a)$, where $h$ runs through $H$. Then $h(a)^g=h(ag)$ and
 $H^G=\gnrsuch{\fY(a)}{\bS(\fY(a)),\ [y(a),z(b)]=1}$, where $a,b\in G,\,a\ne b$. The following lemma is quite evident.
 
 \begin{lemma}\label{l1}
 Suppose a group $\,G$ acting on a group $N$. Let $G=\gnrsuch{\fX}{\bR(\fX)}$,
  $N=\gnrsuch{\fY}{\bS(\fY)}$ be presentations of $\,G$ and $\,N$, and $y^x=w_{xy}(\fY)$ for each
 $x\in\fX,\,y\in\fY$. Then their semidirect product $N\rtimes G$ has a presentation
	\[
	N\rtimes G:=\gnrsuch{\fX,\fY}{\bR(\fX),\,\bS(\fY),\,x^{-1}yx=w_{xy}(\fY)
	 \text{\emph{ for all }} x\in\fX,\,y\in\fY}.
  \] 
 \emph{Note that this presentation may not be minimal even if both presentations for $G$ and $N$ were so,
 since some elements of $\fY$ may become superfluous.}
 \end{lemma}
 
 \begin{corol}\label{c1}
 The wreath product $W=H\wr G$ has indeed a presentation \eqref{e1}.
 \end{corol}
 \begin{proof}
 Lemma~\ref{l1} gives a presentation
 \begin{multline*}
	W:=\gnrsuch{\fX,\,\fY(a)}{\bR(\fX),\,\bS(\fY(a)),\,[y(a),z(b)]=1,\,\\ x^{-1}y(a)x=y(ax)
	 \text{ for } x\in\fX,\,y,z\in\fY,\,a,b\in G,\,a\ne b }.
 \end{multline*}
 Using the last relations, we can exclude all generators $y(a)$ for $a\ne1$; we only have to replace
 $y(a)$ and $z(b)$ by $a^{-1}y(1)a$ and $b^{-1}z(1)b$. So we shall write $h$ instead of $h(1)$ for
 $h\in H$; especially, the relations for $y(a)$ and $z(b)$ are rewritten as $[a^{-1}ya,b^{-1}zb]=1$. 
 The latter is equivalent to $[y,t^{-1}zt]=1$, where $t=ba^{-1}\ne1$. Moreover, the relations $[y,t^{-1}zt]=1$
 and $[z,t yt^{-1}]=1$ are also equivalent; therefore we only need such relations for $t\in T$.
 \end{proof}
 
 \begin{lemma}\label{l2}
 Suppose that $\fY$ is a conormal set of generators of the group $H$, $u,v\in\fY$, and consider the group
 $H_{u,v}=(H*H')/N_{u,v}$, where $*$ denotes the free product of groups, $H'$ is a copy of the group $H$
 whose elements are denoted by $h'\ (h\in H)$, and $N_{u,v}$ is the normal subgroup of $H*H'$ generated
 by the commutators $[y,z']$ with $y,z\in\fY,\,(y,z)\ne(u,v)$. Then $[u,v']\ne1$ in $H_{u,v}$.
 \end{lemma}
 \begin{proof}
 Let $C=H/N_u,\,C'=H'/N_{v'},\,P=C*C'$, $\bar u=uN_u,\ \bar v'=v'N_{v'}$. Consider the homomorphism
 $\phi$ of $H*H'$ to $P$ such that
 \begin{align*}
 \phi(y)&=\begin{cases}
			1 &\text{if } y\in\fY\=\{u\},\\
			\bar u &\text{if } y=u,
     \end{cases} \\
 \phi(z')&=\begin{cases}
			 1 &\text{if } z\in\fY\=\{v\},\\
			 \bar v' &\text{if } z=v.	
				\end{cases}
 \end{align*}
 Obviously, $\phi$ is well defined and $\phi([y,z'])=1$ if $(y,z)\ne (u,v)$, so it induces a homomorphism
 $H_{u,v}\to P$. Since $\phi([u,v'])=[\bar u,\bar v']\ne1$, it accomplishes the proof.
 \end{proof}
 
 Now fix elements $c\in T,\,u,v\in\fY$, and let $K_{c,u,v}$ be the group with a presentation
 \begin{multline*}
  K_{c,u,v}:=\gnrsuch{\fY(a),\ a\in G}{\bS(\fY(a)),\,[y(a),z(ta)]=1 \\
 \text{ for all } y,z\in\fY,\,a\in G,\,t\in T,\,(t,y,z)\ne(c,u,v)}.
 \end{multline*}

 \begin{corol}\label{c2}
 Let the set of generators $\fY$ be conormal. Then
 $[u(1),v(c)]\ne 1$ in the group $K_{c,u,v}$.
 \end{corol}
 \begin{proof}
 There is a homomorphism $\psi:K_{c,u,v}\to H_{u,v}$, where $H_{u,v}$ is the group from Lemma~\ref{l2},
 mapping $u(1)\mapsto u$, $v(c)\mapsto v'$, $y(a)\to1$ in all other cases. Then
 $\psi([u(1),v(c)])=[u,v']\ne1$, so $[u(1),v(c)]\ne1$ as well.
 \end{proof}
  
 \begin{corol}\label{c3}
 If the given presentations of $G$ and $H$ are minimal and the set of generators $\fY$ is conormal,
 the presentation \eqref{e1} is minimal.
 \end{corol}
 \begin{proof}
 Obviously, we can omit from \eqref{e1} neither of generators $\fX,\fY$ nor of the relations $\bR(\fX),\bS(\fY)$.
 So we have to prove that neither relation $[u,c^{-1}vc]=1 \  (u,v\in\fY,\,c\in T)$ can be omitted as well. Consider
 the group $K=K_{c,u,v}$ of Corollary~\ref{c2}. The group $G$ acts on $K$ by the rule: $h(a)^g=h(ag)$. Let
 $Q=K\rtimes G$. Then, just as in the proof of Corollary~\ref{c1}, this group has a presentation 
 \begin{multline*}
	Q:=\gnrsuch{\fX,\fY}{\bR(\fX),\,\bS(\fY),\,[y,t^{-1}zt]=1 \\
	\text{ for all } y,z\in\fY,\,t\in T,\,(t,y,z)\ne(c,u,v)},
 \end{multline*}\
 where $y=y(1)$ for all $y\in\fY$, but $[u,c^{-1}vc]=[u(1),v(c)]\ne 1$.
 \end{proof}
 
 Now for an inductive proof of Theorem~\ref{mult} we only need the following simple result.
 
 \begin{lemma}\label{l3}
 If the sets of generators $\fX$ of $G$ and $\fY$ of $H$ are conormal, so is the set of generators
 $\fX\cup \fY$ of $H\wr G$.
 \end{lemma}
 \begin{proof}
 Since $G\simeq(H\wr G)/\hat H$, where $\hat H$ is the normal subgroup generated by all $y\in\fY$,
 it is clear that neither $x\in\fX$ belongs to the normal subgroup generated by $(\fX\=\{x\})\cup\fY$.
 On the other hand, there is an epimorphism $H\wr G\to C\wr G$, where $C=H/N_y$ for some $y\in\fY$;
 in particular, $C\ne\{1\}$ and is generated by the image $\bar y$ of $y$. Since $C$ is commutative, the map
 $C\wr G\to C,\ (f(x),g)\mapsto\prod_{x\in G}f(x)$ is also an epimorphism mapping $\bar y$ to itself.
 The resulting homomorphism $H\wr G\to C$ maps all $x\in\fX$ as well as all $z\in\fY\=\{y\}$ to $1$
 and $y$ to $\bar y\ne1$, which accomplishes the proof.
 \end{proof}
 
 \begin{exam} 
\begin{enumerate}
	\item The wreath product $C_n\wr C_m$, where $C_n$ denotes the cyclic group of order $n$, has a minimal presentation
	\begin{multline*}
	 C_n\wr C_m:= \gnrsuch{x,y}{x^m=1,\,y^n=1,\,[y,x^{-k}yx^k]=1\\ \text{ for } 1\le k\le m/2}.
	\end{multline*}
	(Possibly, $m=\8$ or $n=\8$, then the relation $x^m=1$ or, respectively, $y^n=1$ should be omitted.)
	\item The iterated wreath product $W_{\lst ns}=\bwr_{k=1}^sC_{n_i}$ has a minimal presentation
		\begin{multline*}
	 W_{\lst ns} := \gnrsuch{\lst xs}{x_i^{n_i}=1,\,[x_i,x_l^{-k}x_jx_l^k]=1\\
	  \text{ for } i<l,\,j<l,\ 1\le k\le n_l/2}.
 		\end{multline*}
\end{enumerate}
 \end{exam}
 
 Especially, since, by the Kaloujnine theorem \cite{ka} (see also \cite[Theorem~7.27]{ro}, though the bibliographic
 refernce is not correct there), the Sylow subgroup of the symmetric group $\bS_{p^n}$ is isomorphic to the wreath
 product $W_{p,p,\dots, p}$ ($n$ times), we get the following corollary.
 
 \begin{corol}
 The Sylow subgroup of $\bS_{p^n}$ has a minimal presentation
 \begin{multline*}
 \gnrsuch{\lst xn}{x_i^p=1\ (1\le i\le n),\ [x_i,x_l^{-k}x_jx_l]=1 \\
 					\text{ \emph{for} } i<l,\,j<l,\ 1\le k\le p/2}.
 \end{multline*}
 \emph{(Certainly, if $p>2$, $k\le p/2$ actually means $k\le(p-1)/2$.)}
 \end{corol}

 \end{document}